\documentclass[11pt]{article}

\usepackage{amsfonts}
\usepackage{amscd}
\usepackage{amssymb}
\usepackage{amsthm}
\usepackage{amsmath}
\usepackage{stmaryrd}
\usepackage{graphicx}
\usepackage{color}
\usepackage{verbatim}
\usepackage{mathrsfs}
\usepackage{dsfont}

\newcommand{\vect}[1]{\boldsymbol{#1}}

\input prepictex
\input pictex
\input postpictex

 \theoremstyle{plain}
\newtheorem{thm}{Theorem}[section]
\newtheorem{lemma}[thm]{Lemma}
\newtheorem{prop}[thm]{Proposition}
\newtheorem{cor}[thm]{Corollary}

\theoremstyle{definition}

\newtheorem{remark}[thm]{Remark}

\numberwithin{equation}{section}

\def\cC{\mathcal{C}}

\def\cR{\mathcal{R}}

\def\CC{\mathbb{C}}

\def\FF{\mathbb{F}}

\def\NN{\mathbb{N}}

\def\ZZ{\mathbb{Z}}

\def\scH{\mathscr{H}}

\makeatletter
\renewcommand{\@makefnmark}{\mbox{\textsuperscript{}}}
\makeatother

\title{A classification of commutative parabolic Hecke algebras}
\author{Peter Abramenko 
\and
James Parkinson\footnote{Research supported under the Australian Research Council (ARC) discovery grant DP110103205.} 
\and
Hendrik Van Maldeghem}
\date{}

\begin{document}

\maketitle

\begin{abstract} Let $(W,S)$ be a Coxeter system with $I\subseteq S$ such that the parabolic subgroup $W_I$ is finite. Associated to this data there is a \textit{Hecke algebra} $\scH$ and a \textit{parabolic Hecke algebra} $\scH^I=\mathbf{1}_I\scH\mathbf{1}_I$ (over a ring $\ZZ[q_s]_{s\in S}$). We give a complete classification of the commutative parabolic Hecke algebras across all Coxeter types.
\end{abstract}

\section*{Introduction}

\textit{Parabolic Hecke algebras} $\scH^I$ arise naturally as algebras of $P_I$ bi-invariant functions on semisimple Lie (or Kac-Moody) groups $G$ defined over finite fields, where $P_I$ is a type $I$ parabolic subgroup. As such they play an important role in the representation of these groups, in particular in studying the representations which have a $P_I$-fixed vector. If $\scH^I$ is commutative then $(G,P_I)$ is a \textit{Gelfand pair}. In this case the representation theory of $\scH^I$ is considerably simplified, and this leads to powerful results about representations of the group $G$. See, for example, \cite{borel}, \cite{macdonald} and \cite{matsumoto} for the affine case. Thus it is a natural question to ask when these algebras are commutative. 

Hecke algebras can be defined more generally, without reference to Kac-Moody groups as follows. Let $(W,S)$ be a Coxeter system, and let $(q_s)_{s\in S}$ be a family of commuting indeterminants with $q_s=q_t$ if and only if~$s$ and~$t$ are conjugate in~$W$. The \textit{Hecke algebra} is the associative $\ZZ[q_s]_{s\in S}$ algebra $\scH$ with free basis $\{T_w\mid w\in W\}$ and relations given by equations (\ref{eq:rel}) in Section~\ref{section:hecke}. Suppose that $I\subseteq S$ is such that the parabolic subgroup $W_I=\langle \{s\mid s\in I\}\rangle$ is finite. The \textit{$I$-parabolic Hecke algebra} $\scH^I$ is 
$$
\scH^I=\mathbf{1}_I\scH\mathbf{1}_I,\qquad\textrm{where}\qquad \mathbf{1}_I=\sum_{w\in W_I}T_w.
$$
It is these algebras (and their specialisations with $q_s\geq 1$) that we study here. We give a complete classification of the pairs $(W,I)$ with $W$ irreducible such that $\scH^I$ is commutative.

Let us put this result into perspective by surveying known results on the commutativity of parabolic Hecke algebras. Assume throughout that $W$ is irreducible. Consider the \textit{spherical case} (that is, $|W|<\infty$). The case $|S\backslash I|=1$ (that is, $W_I$ is a maximal parabolic subgroup of $W$) is classical, dating back to Iwahori~\cite{Iwahori} with proofs appearing in~\cite{CIK} (see also \cite[Theorem~10.4.11]{BCN}). It turns out that the statement is very neat in this case: $\scH^I$ is commutative if and only if each minimal length $W_I$ double coset representative is an involution. This statement does not hold in general (however we obtain a similar equivalence in Theorem~\ref{thm:commutativity}). The proof in \cite{CIK} uses elegant representation theory of the Coxeter group~$W$, along with counting arguments, semisimplicity of the Hecke algebra, and Tits' Deformation Theorem. These techniques do not readily generalise to the infinite case, as we lose the counting arguments, semisimplicity, and the Deformation Theorem.

The spherical case with $|S\backslash I|=1$ is also analysed in \cite{lehrer} via incidence structures and permutation representations. In particular \cite[Section~4]{lehrer} gives a thorough analysis of the classical types, and in \cite[Section~6]{lehrer} the question of studying the spherical case with $|S\backslash I|>1$ is raised. It is shown in \cite[Lemma~III.3.5]{krieg} that if $W$ is of type $A_n$ and $|S\backslash I|>1$ then $\scH^I$ is noncommutative. The main result in \cite{Anderson} extends this to show that if $W$ is spherical and $|S\backslash I|>1$ then $\scH^I$ is noncommutative. We give a very short proof of this fact across \textit{all} Coxeter types in Section~\ref{section:proof} (it appears to have been previously known only for the spherical types via a case by case argument involving computer calculations for the exceptional types).

Now suppose that $W$ is \textit{affine} (see Section~\ref{sect:coxeter}). 
If $I=S\backslash\{i\}$ with $i$ a \textit{special vertex} then it is well known that $\scH^I$ is commutative. This result is important in the representation theory of semisimple Lie groups defined over local fields such as the $p$-adics (see \cite{macdonald}, \cite{matsumoto}). The question of whether commutative parabolic Hecke algebras exist in the affine case with~$i$ not a special vertex is natural, yet to our knowledge has not been treated in the literature. It follows from our classification that there are in fact no such commutative parabolic Hecke algebras.

Now consider the case that $W$ is non-affine and infinite. In \cite[Theorem~3.5]{lecureux} it is shown that maximal parabolic Hecke algebras arising from group actions on locally finite thick buildings of type~$W$ are noncommutative. (However there is a mistake in the proof which needs to be fixed. L\'{e}cureux's Lemma~3.4 only holds for simple reflections, but is used for general reflections in the proof of his Theorem~3.5.) Such buildings can only exist if $m_{st}\in\{2,3,4,6,8,\infty\}$ for each $s,t\in S$ because the Feit-Higman Theorem restricts the possible rank~2 residues. If $W$ is \textit{crystallographic} (that is, $m_{st}\in\{2,3,4,6,\infty\}$, cf. \cite[p.25]{kumar}) then existence of such a building is guaranteed via Kac-Moody theory.

%, and then \cite[Theorem~3.5]{lecureux} shows that if~$|S\backslash I|=1$ then $\scH^I$ is noncommutative.

%We note that there appears to be an error in the proof of \cite[Theorem~3.5]{lecureux} (Lemma~3.4 holds only for simple reflections, but it is used for general reflections in the proof of Theorem~3.5) although we believe that the author has corrected this (private communication). 

In summary, it appears that the following cases are not treated in the literature: (i) $|S\backslash I|>1$ (for general Coxeter types), (ii) the affine case with $I=S\backslash\{i\}$ and $i$ non-special, and (iii) the non-crystallographic non-affine infinite cases. It also appears that the existing techniques do not readily generalise to treat these cases. In this paper we give a systematic and complete classification of commutative parabolic Hecke algebras. Our proof uses a uniform technique to cover all cases (including the known cases). As a consequence it turns out that the three cases listed above give noncommutative parabolic Hecke algebras.

Let us briefly outline the structure of this paper. Section~\ref{sect:1} gives standard definitions and background on Coxeter groups and Hecke algebras, and in Section~\ref{sect:2} we state our classification theorem (Theorem~\ref{thm:list}). We also develop some elementary tests for commutativity and noncommutativity that will be used in Section~\ref{section:proof}, where we give the proof of the classification theorem. The proof has two parts. First we prove that those cases listed in Theorem~\ref{thm:list} give rise to commutative parabolic Hecke algebras. This is achieved using Lemma~\ref{lem:commutativity}, which is inspired by the statement of \cite[Theorem~3.1]{CIK}. Next we show that all remaining cases are noncommutative. This involves some Coxeter graph combinatorics to reduce the analysis to a finite number of cases. In each of these cases a word in the Coxeter group is exhibited, which when fed into our noncommutativity test (Proposition~\ref{prop:test}) proves that the parabolic Hecke algebra is noncommutative. We note that in order to apply our word arguments and diagram combinatorics to the general infinite cases, it is in fact necessary to give our elementary proof of the known noncommutative spherical cases. In the appendix we prove a lemma, make some comments on the structure of double cosets, and list the words we used to deduce noncommutativity.

The majority of the research presented in this paper was conducted in Ghent, Belgium, where the first two authors visited the third author on two occasions. We thank Ghent University for its hospitality. The second author thanks the Australian Research Council for its support under the ARC discovery grant DP110103205. Finally we would like to thank Bob Howlett for useful conversations regarding the results in~\cite{CIK} and for his help using his Coxeter group magma package which was useful in our investigations (however note that our proof does not require computer calculations).

\section{Definitions}\label{sect:1}

This section recalls some standard definitions and results on Coxeter groups, Hecke algebras, and specialisations of Hecke algebras. Standard references include \cite{AB}, \cite{bourbaki}, \cite{humphreys}, and \cite{lusztig}.

\subsection{Coxeter groups}\label{sect:coxeter}

A \textit{Coxeter system} $(W,S)$ is a group $W$ generated by a set $S$ with relations
$$
(st)^{m_{st}}=1\qquad\textrm{for all $s,t\in S$},
$$
where $m_{ss}=1$ and $m_{st}\in\mathbb{Z}_{\geq2}\cup\{\infty\}$ for all $s\neq t$. If $m_{st}=\infty$ then it is understood that there is no relation between $s$ and $t$. We will always assume that $|S|$ is finite. The \textit{Coxeter matrix} of $(W,S)$ is $M=(m_{st})$.

The \textit{length} $\ell(w)$ of $w\in W$ is 
$$
\ell(w)=\min\{n\in\NN\mid w=s_1\cdots s_{n}\textrm{ with } s_1,\ldots,s_{n}\in S\}.
$$
An expression $w=s_1\cdots s_{n}$ with $n=\ell(w)$ is called a \textit{reduced expression} for $w$.

The \textit{Coxeter graph} (or \textit{Coxeter diagram}) of $(W,S)$ is the graph with vertex set $S$ and with $s,t\in S$ joined by an edge if and only if $m_{st}\geq 3$. If $m_{st}\geq 4$ then the corresponding edge is labelled by $m_{st}$. A Coxeter system $(W,S)$ is \textit{irreducible} if its Coxeter graph is connected.

Finite Coxeter groups are called \textit{spherical Coxeter groups}. These are precisely the Coxeter groups whose Coxeter matrix $M$ is positive definite. The irreducible spherical Coxeter groups are classified (see \cite{coxeter}, \cite{bourbaki}, \cite{humphreys}). 

Coxeter groups which are not finite but contain a normal abelian subgroup such that the corresponding quotient group is finite are called \textit{affine Coxeter groups}. These are precisely the Coxeter groups whose Coxeter matrix is positive semidefinite but not positive definite. The irreducible affine Coxeter groups are classified (see \cite{bourbaki}, \cite{humphreys}). In each case the Coxeter graph of an irreducible affine Coxeter group is obtained from the Coxeter matrix of an irreducible spherical Coxeter graph by adding one extra vertex (usually labelled $0$). The vertices of the affine Coxeter graph which are in the orbit of $0$ under the action of the group of diagram automorphisms are called the \textit{special vertices}.

When it is necessary to fix a labelling of the generators of a spherical or affine Coxeter group we will adopt the conventions from~\cite{bourbaki}. 
The \textit{Bruhat partial order} $\leq$ on a Coxeter system $(W,S)$ can be described as follows. If $v,w\in W$ then $v\leq w$ if and only if there is a reduced expression $w=s_1\cdots s_n$ such that $v$ is equal to a \textit{subexpression} of $s_1\cdots s_n$ (that is, an expression obtained by deleting factors). If $v\leq w$ then $v$ is equal to a subexpression of \textit{every} reduced expression of~$w$.
The \textit{deletion condition} says that if $w=s_1\cdots s_n$ with $n>\ell(w)$ then there exists indices $i<j$ such that $w=s_1\cdots\hat{s_i}\cdots\hat{s_j}\cdots s_n$, where $\hat{s}$ indicates that the factor $s$ is omitted.

For $I\subseteq S$ let
$W_I$ be the subgroup of $W$ generated by $I$. Each double coset $W_IwW_I$ has a unique minimal length representative \cite[Proposition~2.23]{AB}. This representative is called \textit{$I$-reduced}, and we let
$$
R_I=\{w\in W\mid w\textrm{ is $I$-reduced}\}.
$$
Thus $R_I$ indexes the decomposition of $W$ into $W_IwW_I$ double cosets.

A subset $I\subseteq S$ is \textit{spherical} if the group $W_I$ is finite. Coxeter systems $(W,S)$ such that there exists a spherical subset $I=S\backslash\{i\}$ are called \textit{nearly finite Coxeter groups} in \cite{howlett}. This class includes the spherical and affine groups, but also many more Coxeter groups.

\subsection{Hecke algebras}\label{section:hecke}

Let $(W,S)$ be a Coxeter system, and let $q_s$, $s\in S$, be commuting indeterminants such that $q_s=q_t$ if and only if $s$ and $t$ are conjugate in~$W$. Let $\cR=\ZZ[q_s]_{s\in S}$ be the polynomial ring in $q_s$, $s\in S$, with integer coefficients. The Hecke algebra $\scH=\scH(W,S)$ is the associative $\cR$-algebra with free basis $\{T_w\mid w\in W\}$ (as an $\cR$-module) and multiplication laws
\begin{align}\label{eq:rel}
T_wT_s=\begin{cases}T_{ws}&\textrm{if $\ell(ws)=\ell(w)+1$}\\
 q_sT_{ws}+(q_s-1)T_w&\textrm{if $\ell(ws)=\ell(w)-1$}.
 \end{cases}
\end{align}

The condition on the parameters implies that the expression
$
q_w=q_{s_1}\cdots q_{s_{\ell}}\in\cR
$
does not depend on the particular choice of reduced expression $w=s_1\cdots s_{\ell}$. 

If $I$ is a spherical subset of $S$ then the element
$$
\mathbf{1}_I=\sum_{w\in W_I}T_w
$$
is in~$\scH$ (since the sum is finite). This element has the following attractive properties, where for finite subsets $X\subseteq W$ the \textit{Poincar\'{e} polynomial} of $X$ is
$
X(q)=\sum_{w\in X}q_w.
$

\begin{lemma}\label{lem:magic}
The element $\mathbf{1}_I$ satisfies $T_w\mathbf{1}_I=\mathbf{1}_IT_w=q_w\mathbf{1}_I$ for all $w\in W_I$, and $\mathbf{1}_I^2=W_I(q)\mathbf{1}_I$.
\end{lemma}

\begin{proof} By induction it suffices to show that $T_s\mathbf{1}_I=\mathbf{1}_IT_s=q_s\mathbf{1}_I$ for each $s\in I$. We have
$$
\mathbf{1}_IT_s=\sum_{w\in W_I}T_wT_s.
$$
Split the sum into two parts, over the sets $W_I^{\pm}=\{w\in W_I\mid \ell(ws)=\ell(w)\pm 1\}$. Using the defining relations (\ref{eq:rel}) and the fact that $W_I^+s=W_I^-$ shows that $\mathbf{1}_IT_s=q_s\mathbf{1}_I$. The $T_s\mathbf{1}_I$ case is similar, using the formula $T_sT_w=q_sT_{sw}+(q_s-1)T_w$ if $\ell(sw)=\ell(w)-1$ (which follows from (\ref{eq:rel})). The fact that $\mathbf{1}_I^2=W_I(q)\mathbf{1}_I$ follows immediately.
\end{proof}

The \textit{structure constants} $c_{u,v;w}\in\ZZ[q_s]_{s\in S}$ of $\scH$ relative to the basis $\{T_w\mid w\in W\}$ are defined by the equations
\begin{align}
\label{eq:stct}T_uT_v=\sum_{w\in W}c_{u,v;w}T_w\qquad\textrm{for all $u,v\in W$}.
\end{align}

\begin{lemma}\label{lem:cst1} The structure constants $c_{u,v;w}$ are polynomials in $\{q_s-1\mid s\in S\}$ with nonnegative integer coefficients.
\end{lemma}

\begin{proof} 
Induction on $\ell(v)$, with $\ell(v)=0$ trivial. If $\ell(vs)=\ell(v)+1$ then $T_uT_{vs}=(T_uT_v)T_s$. Expanding the left hand side of this equation using (\ref{eq:stct}) and the right hand side using (\ref{eq:stct}) and (\ref{eq:rel}) gives
$$
c_{u,vs;w}=\begin{cases}
c_{u,v;ws}q_s&\textrm{if $\ell(ws)=\ell(w)+1$}\\
c_{u,v;ws}+c_{u,v;w}(q_s-1)&\textrm{if $\ell(ws)=\ell(w)-1$}.
\end{cases}
$$
By the induction hypothesis $c_{u,v;w}$ and $c_{u,v;ws}$ are polynomials in $\{q_s-1\mid s\in S\}$ with nonnegative integer coefficients, and so $c_{u,vs;w}$ is too (since $q_s=1+(q_s-1)$).
\end{proof}

\subsection{Parabolic Hecke algebras}

Let $\scH$ be the Hecke algebra with Coxeter system $(W,S)$ and let $I\subseteq S$ be spherical. The \textit{$I$-parabolic Hecke algebra} is 
$$
\scH^I=\mathbf{1}_I\scH\mathbf{1}_I.
$$
We note that in general $\scH^I$ is not unital (as $W_I(q)$ is not an invertible element of $\ZZ[q_s]_{s\in S}$).

Let $I$ be spherical and let $w\in R_I$ be $I$-reduced. We define
$$
T_w^I=\frac{W_I(q)}{W_{I\cap wIw^{-1}}(q)}\mathbf{1}_IT_w\mathbf{1}_I.
$$
The Poincar\'{e} polynomial $W_I(q)$ is divisible by $W_{I\cap wIw^{-1}}(q)$ (this follows from equation (\ref{eq:stabiliser}) below and statement (a) immediately following (\ref{eq:stabiliser})), and so 
the quotient is really an element of the coefficient ring $\cR=\ZZ[q_s]_{s\in S}$.

The set $\{T_w^I\mid w\in R_I\}$ is a linear basis for $\scH^I$ (Proposition~\ref{prop:c}). Let $c_{u,v;w}^I$, $u,v,w\in R_I$, be the structure constants of $\scH^I$ relative to this basis, defined by the equations
$$
T_u^IT_v^I=\sum_{w\in R_I}c_{u,v;w}^IT_w^I\qquad\textrm{for $u,v\in R_I$}.
$$
If $I=\emptyset$ then $\mathbf{1}_I=1$ (the identity in $\scH$), and so $T_w^I=T_w$ and $\scH^I=\scH$. Thus $c_{u,v;w}^{\emptyset}=c_{u,v;w}$ are the structure constants appearing in~(\ref{eq:stct}). Part (ii) of the following proposition relates the structure constants $c_{u,v;w}^I$ to the more elementary structure constants $c_{u,v;w}$.

\begin{prop}\label{prop:c} Let $I\subseteq S$ be spherical.
\begin{enumerate}
\item[(i)] For $w\in R_I$ we have
$$
T_w^{I}=W_I(q)\sum_{z\in W_IwW_I}T_z,
$$
and $\{T_w^I\mid w\in R_I\}$ is a linear basis for $\scH^I$.
\item[(ii)] Let $u,v,w\in R_I$. For any $z\in W_IwW_I$ we have
\begin{align*}
c_{u,v;w}^I=W_I(q)\sum_{\substack{x\in  W_IuW_I\\
y\in W_IvW_I}}c_{x,y;z}.
\end{align*}
\end{enumerate}
\end{prop}

\begin{proof} Let $W_{I,w}$ be the subgroup of $W_I$ stabilising $wW_I$, and let $M_{I,w}$ be a fixed set of minimal length representatives of cosets in $W_I/W_{I,w}$. Notice that $s\in S\cap W_{I,w}$ if and only if $s\in W_I$ and $s\in wW_Iw^{-1}$, and hence (see \cite[Lemma~2.25]{AB})
\begin{align}\label{eq:stabiliser}
W_{I,w}=W_I\cap wW_Iw^{-1}=W_{I\cap wIw^{-1}}.
\end{align}
If $w\in R_I$ then (see \cite[\S2.3.2]{AB})
\begin{enumerate}
\item[(a)] Each $u\in W_I$ can be written in exactly one way as 
		$u=xy$ with $x\in M_{I,w}$ and $y\in W_{I,w}$. Moreover 
	$
	\ell(u)=\ell(x)+\ell(y)$ for any such expression.
\item[(b)] Each $v\in W_IwW_I$ can be written in exactly one way as
	$
	v=xwy$ with $x\in M_{I,w}$ and $y\in W_I$.
	Moreover $\ell(v)=\ell(x)+\ell(w)+\ell(y)$ for any such expression.
	\end{enumerate}
Using (a) we have
\begin{align*}
\mathbf{1}_IT_w\mathbf{1}_I&=\sum_{u\in W_I}T_uT_w\mathbf{1}_I=\sum_{x\in M_{I,w}}\sum_{y\in W_{I,w}}T_xT_yT_w\mathbf{1}_I.
\end{align*}
Since $w$ is $I$-reduced we have
$\ell(yw)=\ell(y)+\ell(w)$ for each $y\in W_{I,w}$, and $yw=wy'$ for some $y'\in W_{I,w}$ with $\ell(wy')=\ell(w)+\ell(y')$. This implies that $q_{y'}=q_y$, and (\ref{eq:rel}) and Lemma~\ref{lem:magic} give
$$
T_yT_w\mathbf{1}_I=T_{yw}\mathbf{1}_I=T_{wy'}\mathbf{1}_I=T_wT_{y'}\mathbf{1}_I=q_yT_w\mathbf{1}_I.
$$
Thus by (\ref{eq:stabiliser}) we have $\sum_{y\in W_{I,w}}T_xT_yT_w\mathbf{1}_I=W_{I\cap wIw^{-1}}(q)T_xT_w\mathbf{1}_I$, and hence by (b) we compute
$$
T_w^{I}=W_I(q)\sum_{x\in M_{I,w}}T_xT_w\mathbf{1}_I=W_I(q)\sum_{x\in M_{I,w}}\sum_{y\in W_I}T_xT_wT_y=W_I(q)\sum_{z\in W_IwW_I}T_{z}.
$$
This formula shows that $\{T_w^I\mid w\in R_I\}$ is a linearly independent set (since double cosets are either equal or disjoint, and $\{T_w\mid w\in W\}$ is a basis for $\scH$). It also spans $\scH^I$, for if $z\in W$ then $z\in W_IwW_I$ for some $w\in R_I$, and since $w$ is $I$-reduced we have $z=xwy$ with $x\in W_I$, $y\in W_I$, and $\ell(z)=\ell(x)+\ell(w)+\ell(y)$. Then using (\ref{eq:rel}) and Lemma~\ref{lem:magic} we have
$
\mathbf{1}_IT_z\mathbf{1}_I=\mathbf{1}_IT_xT_wT_y\mathbf{1}_I=q_xq_y\mathbf{1}_IT_w\mathbf{1}_I.
$
This completes the proof of (i).

To prove (ii) we use (i) and the expansion $T_xT_y=\sum_z c_{x,y;z}T_z$ to write
\begin{align*}
T_u^{I}T_v^{I}&=W_I(q)^2\sum_{\substack{x\in W_IuW_I\\ y\in W_IvW_I}}T_xT_y=W_I(q)^2\sum_{z\in W}\Bigg(\sum_{\substack{x\in W_IuW_I\\
y\in W_IvW_I}}c_{x,y;z}\Bigg)T_z.
\end{align*}
On the other hand we have
$$
T_u^IT_v^I=\sum_{w\in R_I}c_{u,v;w}^IT_w^I=W_I(q)\sum_{w\in R_I}\bigg(c_{u,v;w}^I\sum_{z\in W_IwW_I} T_z\bigg).
$$
The result follows by comparing coefficients of $T_z$ in these expressions.
\end{proof}

\begin{remark} The structure constants $c_{u,v;w}^I$ in the spherical case are studied in \cite{BC} and~\cite{gomi}. In the affine case formulae are available using \textit{positively folded alcove walks} (see \cite{schwer}). 
\end{remark}

\subsection{Specialisations of the Hecke algebra}

One is often interested in \textit{specialisations} of the Hecke algebra, where the parameters $q_s$, $s\in S$, are chosen to be specific complex numbers. Let us briefly describe this construction. Let $\vect{\tau}=(\tau_s)_{s\in S}$ be a sequence of complex numbers with $\tau_s=\tau_t$ whenever $s$ and $t$ are conjugate in~$W$. Let $\psi:\cR\to\CC$ be the ring homomorphism given by $\psi(q_s)=\tau_s$ for each $s\in S$. Then $\CC$ becomes a $(\CC,\cR)$-bimodule via $(\lambda,\mu,x)\mapsto \lambda\mu\psi(x)$ for all $\lambda,\mu\in\CC$ and $x\in\cR$. The \textit{specialised Hecke algebra} is $\scH_{\vect{\tau}}=\CC\otimes_{\cR}\scH$. This is an algebra over $\CC$ with basis $\{1\otimes T_w\mid w\in W\}$. Note that the specialisation of $\scH$ with $\tau_s=1$ for all $s\in S$ is equal to the group algebra of~$W$.

Let $\scH^I_{\vect{\tau}}$ be the specialisation of $\scH^I$ with parameters $\vect{\tau}=(\tau_s)$. Our classification of commutative parabolic Hecke algebras applies to the `generic' parabolic Hecke algebras $\scH^I$ (defined over $\ZZ[q_s]_{s\in S}$) \textit{and} to the specialisations $\scH^I_{\vect{\tau}}$ with $\tau_s\geq 1$ for all $s\in S$. Potential problems arise for other values of $\tau_s$, since our argument in Corollary~\ref{cor:comcondition}, which relies on Corollary~\ref{cor:cst1} below, breaks down.

The structure constants of the specialised algebra $\scH^I_{\vect{\tau}}$ are obtained by applying the evaluation homomorphism $\psi:\ZZ[q_s]_{s\in S}\to\CC$ with $\psi(q_s)=\tau_s$ to the structure constants of the generic algebra~$\scH^I$. 

\begin{cor}\label{cor:cst1} If $\tau_s\geq 1$ for all $s\in S$ then $\psi(c_{u,v;w}^I)\geq0$, and if the constant term of $c_{u,v;w}$ when written as a polynomial in the variables $q_s-1$ is nonzero then $\psi(c_{u,v;w}^I)>0$.
\end{cor}

\begin{proof}
By Lemma~\ref{lem:cst1} the claim is true for $I=\emptyset$ (where $c_{u,v;w}^I=c_{u,v;w}$), and by Proposition~\ref{prop:c} we see that the claim holds for general (spherical) $I$, since $W_I(\vect{\tau})>0$ if $\tau_s\geq1$ for all $s\in S$.
\end{proof}

\begin{remark}\label{rem:kac}
If $\tau_s=p^n$ for all $s\in S$ with $p$ a prime then $\scH_{\vect{\tau}}\cong \cC_c(B\backslash G/B)$. Here $G$ is a Kac-Moody group of type $W$ over the finite field~$\FF_{p^n}$ (see \cite{titskac2}), $B$ is the standard Borel subgroup of~$G$, and $\cC_c(B\backslash G/B)$ is the convolution algebra of $B$ bi-invariant functions $f:G\to\CC$ supported on finitely many $B$ double cosets. For such a Kac-Moody group to exist it is necessary and sufficient that $m_{st}\in\{2,3,4,6,\infty\}$ for each $s,t\in S$ (see \cite[Proposition~1.3.21]{kumar}). Similarly $\scH_{\vect{\tau}}^I\cong \cC_c(P_I\backslash G/P_I)$ where $P_I$ is the standard $I$-parabolic subgroup $P_I=\bigsqcup_{w\in W_I}BwB.$ 
\end{remark}

\begin{remark}
Suppose that $\tau_s=\tau$ for all $s\in S$. If $W$ is spherical then $\scH_{\vect{\tau}}$ is isomorphic to the group algebra of~$W$ for all values of $\tau\in\CC^{\times}$ except for roots of the Poincar\'{e} polynomial $W(\tau)$ \cite[\S68A]{CR}. This statement is usually \textit{not} true for infinite Coxeter groups~$W$ (see~\cite[\S11.7]{Xi}).
\end{remark}

\section{Commutativity of $\scH^I$}\label{sect:2}

\subsection{Statement of results}\label{sect:statement}

The following classification theorem is the main result of this paper. The proof is given in the next section after giving some preliminary observations in this section. We use Bourbaki \cite{bourbaki} conventions for the labelling of the nodes of spherical and affine Coxeter systems. In the $H_3$ and $H_4$ cases (where there is no explicit labelling given in \cite{bourbaki}) we take $m_{12}=3$ and $m_{23}=5$ in the $H_3$ case, and $m_{12}=m_{23}=3$ and $m_{34}=5$ in the $H_4$ case.

If $X_n$ is a spherical Coxeter diagram and if $i$ is a vertex of $X_n$  then we write $X_{n,i}$ to denote the case where $(W,S)$ has type $X_n$ and $I=S\backslash\{i\}$. Similarly if $\tilde{X}_n$ is an affine diagram then the notation $\tilde{X}_{n,i}$ means that $(W,S)$ has type $\tilde{X}_n$ and $I=S\backslash\{i\}$.

\begin{thm}\label{thm:list} Let $(W,S)$ be irreducible, let $I\subseteq S$ be spherical, and let $\vect{\tau}=(\tau_s)$ with $\tau_s\geq 1$ for each $s\in S$. The $I$-parabolic Hecke algebras $\scH^I$ and $\scH_{\vect{\tau}}^I$ are noncommutative if $|S\backslash I|>1$. If $I=S\backslash\{i\}$ then $\scH^I$ and $\scH_{\vect{\tau}}^I$ are commutative in the cases
\begin{enumerate}
\item[$\bullet$] $A_{n,i}$ ($1\leq i\leq n$), $B_{n,i}$ ($1\leq i\leq n$), $D_{n,i}$ ($1\leq i\leq   n/2$ or $i=n-1,n$), $E_{6,1}$, $E_{6,2}$, $E_{6,6}$, $E_{7,1}$, $E_{7,2}$, $E_{7,7}$, $E_{8,1}$, $E_{8,8}$, $F_{4,1}$, $F_{4,4}$, $H_{3,1}$, $H_{3,3}$, $H_{4,1}$, $I_{2}(p)_{i}$ ($i=1,2$), and
\item[$\bullet$] all affine cases $\tilde{X}_{n,i}$ with $i$ a special type,
\end{enumerate}
and noncommutative otherwise. 
\end{thm}

As a consequence of this classification it turns out that we have the following uniform statement which has the same flavour as  \cite[Theorem~3.1]{CIK}. The idea is modeled on \cite[Theorem~5.21 and Theorem~5.24]{P1}. 

\begin{thm}\label{thm:commutativity} With the notation of Theorem~\ref{thm:list}, the algebras $\scH^I$ and $\scH^I_{\vect{\tau}}$ are commutative if and only if there is an automorphism $\pi$ of the Coxeter diagram such that
\begin{enumerate}
\item[(a)] $\pi(I)=I$, 
\item[(b)] $\pi(w)=w^{-1}$ for all $w\in R_I$, and
\item[(c)] $q_{\pi(s)}=q_s$ for all $s\in S$.
\end{enumerate}
\end{thm}

\begin{remark} Suppose that the Coxeter system $(W,S)$ is not irreducible. Let $S=S_1\cup\cdots \cup S_n$ be the decomposition of the nodes of the Coxeter graph into connected components, and let $W_j=\langle S_j\rangle$ for each $j=1,\ldots,n$. It is elementary that
$$
\scH(W,S)\cong \scH(W_1,S_1)\oplus\cdots\oplus\scH(W_n,S_n).
$$
Let $I\subseteq S$ be spherical, and let $I_j=I\cap S_j$. Then $\mathbf{1}_I=\mathbf{1}_{I_1}\cdots\mathbf{1}_{I_n}$, and it follows that 
$$
\scH^I(W,S)\cong \scH^{I_1}(W_1,S_1)\oplus\cdots\oplus\scH^{I_n}(W_n,S_n).
$$
Thus $\scH^I(W,S)$ is commutative if and only if each $\scH^{I_j}(W_j,S_j)$ is commutative. Thus we will henceforth assume that the $(W,S)$ is irreducible. 
\end{remark}

\begin{remark} 
In the spherical case (except for $H_3$ and $H_4$) commutativity of $X_{n,i}$ is dealt with in \cite[Theorem~3.1]{CIK} (see also \cite[Theorem~10.4.11]{BCN}). We give a different elementary proof here. In fact our proof technique for the general case makes it crucial for us to give our proof of the spherical case.
\end{remark}

\subsection{Initial observations}

By induction on $\ell(y)$ we see that $c_{x,y;z}=c_{y^{-1},x^{-1};z^{-1}}$, and so by Proposition~\ref{prop:c} we see that 
\begin{align}\label{eq:commobservation}
c_{v^{-1},u^{-1};w^{-1}}^I&=W_I(q)\sum_{\substack{x\in W_IvW_I\\
y\in W_IuW_I}}c_{x^{-1},y^{-1};z^{-1}}=W_I(q)\sum_{\substack{
y\in W_IuW_I\\
x\in W_IvW_I}}c_{y,x;z}=c_{u,v;w}^I,
\end{align}
where $z$ is any element of the double coset $W_Iw^{-1}W_I$. Thus if each $w\in R_I$ is an involution then $c_{u,v;w}^I=c_{v,u;w}^I$, and so the algebra $\scH^I$ is commutative. It turns out that in the spherical case this is an equivalence: $\scH^I$ is commutative if and only if each element of $R_I$ is an involution (see \cite[Theorem~3.1]{CIK} and Claim~1 in Section~\ref{section:proof} below). However it is not an equivalence in arbitrary type (as the affine cases with special vertices show).

\begin{lemma}\label{lem:commutativity} Suppose that there is an automorphism $\pi$ of the Coxeter graph satisfying conditions (a), (b) and (c) of Theorem~\ref{thm:commutativity}. Then the algebras $\scH^I$ and $\scH_{\vect{\tau}}^I$ (for any specialisation $\tau_s\in\CC$) are commutative.
\end{lemma}

\begin{proof} We claim that the property $q_{\pi(s)}=q_s$ implies that
\begin{align}\label{eq:pi}
c_{x,y;z}=c_{\pi(x),\pi(y);\pi(z)}\qquad\textrm{for all $x,y,z\in W$}.
\end{align}
We argue by induction on $\ell(y)$, with $\ell(y)=0$ trivial. If $\ell(sy)>\ell(y)$, then expanding $T_xT_{sy}=(T_xT_s)T_y$ in two ways using (\ref{eq:rel}) gives
$$
c_{x,sy;z}=\begin{cases} c_{xs,y;z}&\textrm{if $\ell(xs)>\ell(x)$}\\
q_sc_{xs,y;z}+(q_s-1)c_{x,y;z}&\textrm{if $\ell(xs)<\ell(x)$}.
\end{cases}
$$
By the induction hypothesis and property (c) we have
$
c_{x,sy;z}=c_{\pi(x),\pi(sy);\pi(z)},
$
hence (\ref{eq:pi}).

By properties (a) and (b) if $w$ is $I$-reduced then
$
\pi(W_IwW_I)=W_Iw^{-1}W_I=(W_IwW_I)^{-1}.
$
Using this observation, by Proposition~\ref{prop:c} and (\ref{eq:pi}) we have $
c_{\pi(u),\pi(v);\pi(w)}^I=c_{u,v;w}^I$.

On the other hand, by (b) and (\ref{eq:commobservation}) we have
$c_{\pi(u),\pi(v);\pi(w)}^{I}=c_{u^{-1},v^{-1};w^{-1}}^I=c_{v,u;w}^I$.
Thus  $c_{u,v;w}^I=c_{v,u;w}^I$. So $\scH^I$ is commutative, and hence $\scH_{\vect{\tau}}^I$ is commutative for each specialisation.
\end{proof}

\begin{lemma}\label{lem:comcondition} Let $u,v,w\in R_I$. If $c_{u,v;w}^I\neq 0$ then $w\in W_Iu'W_IvW_I$ for some $u'\leq u$.
\end{lemma}

\begin{proof}
Recall that $T_u^I$ is a scalar times $\mathbf{1}_IT_u\mathbf{1}_I$. Thus $T_u^IT_v^I=\sum c_{u,v;w}^IT_w^I$ is a scalar times
\begin{align}\label{eq:re}
\mathbf{1}_IT_u\mathbf{1}_I\cdot\mathbf{1}_IT_v\mathbf{1}_I=W_I(q)\mathbf{1}_IT_u\mathbf{1}_IT_v\mathbf{1}_I=W_I(q)\sum_{z\in W_I}\mathbf{1}_IT_uT_zT_v\mathbf{1}_I.
\end{align}
Since $v\in R_I$ we have $T_zT_v=T_{zv}$ for each $z\in W_I$. An induction on $\ell(u)$ using (\ref{eq:rel}) shows that $T_{u}T_{zv}$ is a linear combination of terms $T_{u'zv}$ with $u'\leq u$. Therefore the right hand side of (\ref{eq:re}) is a linear combination of terms $\{\mathbf{1}_IT_x\mathbf{1}_I\mid x\in u'W_Iv, u'\leq u\}$. It follows from Lemma~\ref{lem:magic} that for each $x\in W$, $\mathbf{1}_IT_x\mathbf{1}_I$ is a nonzero scalar multiple of $\mathbf{1}_IT_{x'}\mathbf{1}_I$, where $x'$ is the unique $I$-reduced element of $W_IxW_I$ (see the proof of Proposition~\ref{prop:c}). Therefore the right hand side of (\ref{eq:re}) is a linear combination of terms $\mathbf{1}_IT_{x'}\mathbf{1}_I$ with $x'$ being the $I$-reduced element of a double coset of the form $W_Iu'W_IvW_I$ with $u'\leq u$. The result follows. 
\end{proof}

Thus we obtain the following general test for noncommutativity.

\begin{cor}\label{cor:comcondition} Let $u,v,w\in R_I$. Suppose that $w=uzv$ with $\ell(w)=\ell(u)+\ell(z)+\ell(v)$ and $z\in W_I$. If there does not exist $u',v',z'$ with $u'\leq u$, $v'\leq v$, and $z'\in W_I$ such that $w=v'z'u'$ and $\ell(w)=\ell(v')+\ell(z')+\ell(u')$, then $\scH^I$ and $\scH_{\vect{\tau}}^I$ (with $\tau_s\geq 1$) are noncommutative.
\end{cor}

\begin{proof}
Let $\psi:\ZZ[q_s]_{s\in S}\to\CC$ be the evaluation homomorphism with $\psi(q_s)=\tau_s\geq 1$ for each $s\in S$.
We claim that if $w=uzv$ with $z\in W_I$ and $\ell(w)=\ell(u)+\ell(z)+\ell(v)$ then $c_{u,v;w}^I\neq 0$ and $\psi(c_{u,v;w}^I)>0$. To see this, note that by Proposition~\ref{prop:c} and the defining relations (\ref{eq:rel}) we have
\begin{align*}
c_{u,v;uzv}^I&=W_I(q)\big(c_{uz,v;uzv}+\textrm{positive linear combination of other $c_{x,y;z}$ terms}\big)\\
&=W_I(q)\big(1+\textrm{positive linear combination of other $c_{x,y;z}$ terms}\big),
\end{align*}
from which the result follows (see Lemma~\ref{lem:cst1} and Corollary~\ref{cor:cst1}).

Next we claim that under the assumptions of the Corollary we have $c_{v,u;w}^I=0$ (and hence $\psi(c_{u,v;w}^I)=0$ too). For if $c_{v,u;w}^I\neq 0$ then by Lemma~\ref{lem:comcondition} we have $w\in W_Iv'W_IuW_I$ for some $v'\leq v$, and so $w=w_1v'w_2uw_3$ with $w_1,w_2,w_3\in W_I$. By repeated applications of the deletion condition we obtain a reduced word $w=w_1'v''w_2'u'w_3'$ with $w_1',w_2',w_3'\in W_I$ and $v''\leq v$ and $u'\leq u$. But every reduced expression for an $I$-reduced word starts and ends with elements from $S\backslash I$. Thus $w_1'=w_3'=1$, and so $w=v''w_2'u'$ with $\ell(w)=\ell(v'')+\ell(w_2')+\ell(u')$, contradicting the hypothesis of the corollary.
\end{proof}

The following more specific test for noncommutativity will be used frequently.

\begin{prop}\label{prop:test} Let $I=S\backslash\{i\}$. Suppose that there is an element $w\in R_{I}$ such that $w=uw_Ii$ with $u\in R_I$, $w_I\in W_I$, and $\ell(w)=\ell(u)+\ell(w_I)+1$. Fix reduced expressions for $u$ and $w_I$, and suppose that:
\begin{enumerate}
\item[$(1)$] the induced decomposition $w=uw_Ii$ has the minimal number of $i$ factors amongst all possible reduced expressions for $w$, and
\item[$(2)$] there is a generator $k\in I$ that appears in $w_I$ but not in $u$, and that in every reduced expression for $w$ with the minimal number of $i$ factors no occurrence of this $k$ generator appears between the first two $i$ generators of the expression.
\end{enumerate}
Then $\scH^I$ and $\scH_{\vect{\tau}}^I$ (with $\tau_s\geq 1$) are noncommutative.
\end{prop}

\begin{proof} By Corollary~\ref{cor:comcondition} it is sufficient to show that $w$ cannot be written as $w=i'z'u'$ with $i'\in\{\mathrm{id},i\}$, $u'\leq u$, $z'\in W_I$, and $\ell(w)=\ell(i')+\ell(z')+\ell(u')$. Suppose we have such an expression. By (1) we see that $i'=i$, and that $u'$ has the same number of $i$ factors as $u$ does. In particular, $u'$ starts and ends with an $i$. Since $u'$ contains no $k$ factors we see that $z'$ must contain some $k$ factors. Then these factors are between the first two $i$ generators, contradicting~(2).
\end{proof}

\section{Proof of Theorem~\ref{thm:list}}\label{section:proof}

We use the following notation. If $X_n$ is a spherical Coxeter type with nodes $1,2,\ldots,n$ then $X_n^{i}$ is the Coxeter graph obtained by attaching a new node (labelled $0$) to the $i$ node of $X_n$ by a single bond. Similarly, $X_n^{ij}$ with $i\neq j$ indicates that this new node is connected to $i$ and $j$ by single bonds, and $X_n^{ii}$ indicates that $0$ is joined to $i$ by a double bond. This notation naturally extends, and, for example, $F_4^{1,1}\times E_7^{2,5,6}$ indicates that a new node $0$ is connected to the $1$ node of an $F_4$ diagram by a double bond, and to the $2$, $5$ and $6$ nodes of an $E_7$ diagram by single bonds. Also, recall the notation $X_{n,i}$ and $\tilde{X}_{n,i}$ from the beginning of Section~\ref{sect:statement}.

Recall that we assume throughout that $(W,S)$ is irreducible. The proof of Theorem~\ref{thm:list} is achieved via the following 6 claims. The first claim shows that if $|S\backslash I|>1$ then $\scH^I$ is noncommutative, allowing us to focus on the maximal parabolic case $I=S\backslash\{i\}$. The second and third claims deal with the commutative spherical and affine cases. In claim 4 we produce a list of noncommutative cases. This library of noncommutative cases is used in claims 5 and 6 to show that all cases other than those listed in Theorem~\ref{thm:list} are noncommutative.

\bigskip

\noindent\textbf{Claim 1:} \textit{If $|S\backslash I|>1$ then $\scH^I$ and $\scH_{\vect{\tau}}^I$ (with $\tau_s\geq 1$) are noncommutative.}

\begin{proof} Choose vertices $s,t\in S\backslash I$ with $s\neq t$ at minimal length in the (connected) Coxeter graph of~$W$. Then $s,t\in R_I$, and if $s,s_1,\cdots,s_n, t$ is a minimal length path in the Coxeter diagram then $s_1,\ldots,s_n\in W_I$. The $I$-reduced element $w=ss_1\cdots s_nt$ 
satisfies $\ell(w)=\ell(s)+\ell(s_1\cdots s_n)+\ell(t)$. But $w$ cannot be written as $w=t'z's'$ with $t'\leq t$, $s'\leq s$, $z'\in W_I$, and $\ell(w)=\ell(t')+\ell(z')+\ell(s')$, for there is exactly one reduced expression for $w$, and this reduced expression has one $s$, and one~$t$, and the $t$ is to the right of the~$s$. Thus by Corollary~\ref{cor:comcondition} the algebra $\scH^I$ (and its specialisations with $\tau_s\geq 1$) is noncommutative. (Compare with~\cite{Anderson}).
\end{proof}

\smallskip

\noindent\textbf{Claim 2:} \textit{The spherical cases listed in Theorem~\ref{thm:list} are commutative.}

\begin{proof} It is well-known that in each case listed the minimal length double coset representatives are involutions (see Proposition~\ref{prop:example} for the $E_{8,1}$ example). Thus Lemma~\ref{lem:commutativity} applies (with $\pi$ being trivial), and so the algebras are commutative.
\end{proof}

\smallskip

\noindent\textbf{Claim 3:} \textit{If $I=S\backslash\{i\}$ with $i$ a special node of an affine diagram then $\scH^I$ is commutative.}

\begin{proof}
Let $(W,S)$ be an irreducible Coxeter system of affine type, and let $I=S\backslash\{i\}$, where $i$ is a special type. Then $\scH^I$ (and hence $\scH_{\vect{\tau}}$ for all specialisations) is commutative by Lemma~\ref{lem:commutativity} with the diagram automorphism $\pi$ from that lemma being opposition in the spherical residue. In more detail: We may assume that~$i=0$. Let $Q$ be the coroot lattice of the associated root system, and let
$P$ be the coweight lattice, with dominant cone~$P^+$. Let $W_0=W_{S\backslash\{0\}}$. Then $W\cong Q\rtimes W_0$, and $\{t_{\lambda}\mid \lambda\in Q\cap P^+\}$ is a set of $W_0\backslash W/W_0$ representatives, where $t_{\lambda}$ is the translation by~$\lambda$. So the double cosets satisfy $(W_0t_{\lambda}W_0)^{-1}=W_0t_{\lambda}^{-1}W_0=W_0t_{-\lambda}W_0=W_0t_{\lambda^*}W_0$, where $\lambda^*=-w_0\lambda$, with $w_0$ being the longest element of $W_0$. It follows that the minimal length element $m_{\lambda}$ of $W_0t_{\lambda}W_0$ satisfies $m_{\lambda}^{-1}=m_{\lambda^*}$. Hence the automorphism $\pi$ of the Coxeter diagram given by $\pi(0)=0$ and $\alpha_{\pi(j)}=-w_0\alpha_j$ for $j=1,\ldots,n$ satisfies $\pi(m_{\lambda})=m_{\lambda}^{-1}$ for all $\lambda\in Q\cap P^+$. By construction we have $\pi(I)=I$, and considering the connected affine diagrams we have $q_{\pi(s)}=q_s$ for all $s\in S$. Thus by Lemma~\ref{lem:commutativity} $\scH^I$ is commutative (and hence $\scH_{\vect{\tau}}^I$ is too).
\end{proof}

\noindent\textbf{Claim 4:}  \textit{All of the cases listed in the tables in the appendix are noncommutative.}

\begin{proof} We say that an element $w\in W$ has an \textit{essentially unique expression} if every reduced expression for $w$ is obtained from a given reduced expression of $w$ by a sequence of `commutations' (that is, Coxeter moves of the form $st=ts$). It is routine to check that all of the words in the tables in the appendix have essentially unique expressions, except for the $H_{4,4}$, $\tilde{F}_{4,4}$, $\tilde{E}_{8,1}$ and $H_4^1$ words. These words will be dealt with below. For those words with essentially unique expressions it is easy to check that the triple $(u,w_I,k)$ provided in the table satisfies the hypothesis of Proposition~\ref{prop:test}, except for the $B_2^{1,2}$, $B_4^3$, $E_8^1$, $H_3^{1,1}$, $I_2(5)^{1,1}$ and $I_2(7)^1$ words, and so the associated algebras are noncommutative. 
For example, consider the $D_5^3$ word $w=uw_I0$ with $u=03243120$, $w_I=3543$ and $k=5$. To see that there are no $131\mapsto313$ Coxeter moves available one considers each triple $(1,3,1)$ in the given reduced decomposition for $w$ and verifies that there is no sequence of commutations that make these three generators adjacent. One such triple is $w=03\,\underline{1}\,24\,\underline{3}\,\underline{1}\,2035430$, and it is clear that it is impossible to make the first $\underline{1}$ adjacent to the $\underline{3}$ using commutations. Continuing in this fashion one verifies that this word has an essentially unique expression. It is now clear that the word is reduced and $I$-reduced, and that \textit{every} reduced expression for $w$ has the property that the $k=5$ generator does not appear between the first two $0$ generators. Thus Proposition~\ref{prop:test} applies, and so the algebra is noncommutative.

It remains to deal with the $H_{4,4}$, $\tilde{F}_{4,4}$, $\tilde{E}_{8,1}$, $H_4^1$, $B_2^{1,2}$, $B_4^3$, $H_3^{1,1}$, $I_2(5)^{1,1}$ and $I_2(7)^1$ words (these are marked with a $*$ in the appendix). The $H_{4,4}$ word $w=uw_I4$ with $u=434323434$ and $w_I=123$ has only one possible Coxeter move ($323\mapsto 232$). The only Coxeter move available in the resulting expression $w=4342324341234$ is the move $232\mapsto 323$ taking us back to the original expression. Therefore \textit{every} reduced expression for $w$ is obtained from one of
\begin{align*}
&4343234341234\\
&4342324341234
\end{align*}
by using only commutations. Hence it is clear that the $k=1$ generator can never appear in between the first two $4$ generators of a reduced expression for $w$, and so Proposition~\ref{prop:test} applies.

The $\tilde{F}_{4,4}$ word $w=uw_I4$ with $u=43231234$ and $w_I=3231230123$ has exactly one possible Coxeter move ($343\mapsto 434$). The only Coxeter move in the resulting expression is the one returning us to the original expression. Thus, as in the $H_{4,4}$ case, we readily see that Proposition~\ref{prop:test} (with $k=0$) applies.

Consider the $\tilde{E}_{8,1}$ word $w=134562453413245\underline{676}8054324567813456724563452431$. The only Coxeter move possible initially is the $676\mapsto 767$ move. After making this move we get $w=13456245341324\underline{5}7\underline{6}780\underline{5}4324567813456724563452431$. The only new Coxeter move available is the $565\mapsto 656$ move, giving $w=1345624534132476567804324567813456724563452431$. There are now no new Coxeter moves, and so \textit{every} reduced expression for $w$ is obtained from one of 
\begin{align*}
&1345624534132456768054324567813456724563452431\\
&1345624534132457678054324567813456724563452431\\
&1345624534132476567804324567813456724563452431
\end{align*}
using commutations alone. Thus it is clear that the $0$ generator can never be between the first two $1$ generators, and so Proposition~\ref{prop:test} applies.

The details for the $H_4^1$ word $w=uw_I0$ with $u=012343210$ and $w_I=43423412324341234321$ are as follows. Arguing as above one sees that every reduced expression for $w$ is obtained from one of the following three expression by commuting generators:
\begin{align*}
&012343210434234123243412343210\\
&012343210434234132343412343210\\
&012343210434234321234342343210.
\end{align*}
It follows that every reduced expression for $w$ has at least three $4$s between the last two $0$ generators. Thus there is no reduced expression $w=0zu'$ with $u'\leq u$ and $z\in W_I$ because such an expression has at \textit{most} one $4$ between the last two $0$s. Thus Corollary~\ref{cor:comcondition} proves noncommutativity.

Consider the $B_2^{1,2}$ word $w=uw_I0$ with $u=01210$ and $w_I=212$. This word has \textit{exactly one} reduced expression, and this expression has exactly two $2$s in between the last two $0$ generators. Hence there is no reduced expression of the form $w=0zu'$ with $z\in W_I$ and $u'\leq u$, for each such expression has at most one $2$ between the last two $0$s. Thus Corollary~\ref{cor:comcondition} proves noncommutativity.

Consider the $B_4^3$ word $w=uw_I0$ with $u=03430$ and $w_I=234123$. It is clear that every reduced expression for $w$ has at least one $2$ in between the last two $0$ generators. Thus there is no reduced expression of the form $w=0zu'$ with $z\in W_I$ and $u'\leq u$ (since such expressions have no $2$s in between the last two $0$ generators) and so Corollary~\ref{cor:comcondition} proves noncommutativity.

Consider the $E_8^1$ word $w=uw_I0$ with $u=0134254310$, $w_I=654234567813425436542765431$. This word has an essentially unique expression, and so it is clear that every reduced expression for $w$ has at least two $2$s in between the last two 0 generators. Hence there is no reduced expression of the form $w=0zu'$ with $z\in W_I$ and $u' \leq u$, for each such expression has either zero or one $2$s between the last two $0$s.

Consider the $H_3^{1,1}$ word $w=uw_I0$ with $u=010$ and $w_I=232132321$. Every reduced expression for $w$ has at least two $2$s in between the last two $0$ generators. Thus there is no reduced expression of the form $w=0zu'$ with $z\in W_I$ and $u'\leq u$, and so Corollary~\ref{cor:comcondition} proves noncommutativity. Similarly, for the $I_2(5)^{1,1}$ word $w=uw_I0$ with $u=010$ and $w_I=2121$ every reduced expression for this word has at least one $2$ in between the last two $0$ generators. So Corollary~\ref{cor:comcondition} proves noncommutativity. Finally, every reduced expression for the $I_2(7)^1$ word $w=uw_I0$ with $u=012120$ and $w_I=12121$ has exactly three $1$s in between the last two $0$ generators, and as above, Corollary~\ref{cor:comcondition} proves noncommutativity.
\end{proof}

\smallskip

\noindent\textbf{Claim 5:} \textit{All spherical and affine cases other than those listed in Theorem~\ref{thm:list} are noncommutative.}

\begin{proof} Claim~4 above has provided us with a library of noncommutative examples. We use this library to deal with the remaining cases via the following obvious fact: If $I\subseteq S$ is spherical, and if $S'$ is such that $I\subseteq S'\subseteq S$, and if the parabolic Hecke algebra $\scH^I(W_{S'},S')$ is noncommutative, then $\scH^I(W,S)$ is noncommutative too (and the same holds for specialisations with $\tau_s\geq 1$). This is clear, since the former algebra is a subalgebra of the latter.

It is now straightforward to show that all remaining spherical and affine cases are noncommutative. For example $E_{7,5}$ is noncommutative since the $5$ node of $E_7$ plays the role of the $5$ node in an $E_6$ residue, and $E_{6,5}$ is noncommutative by our library. Similarly $\tilde{E}_{8,2}$ is noncommutative since the $2$ node of $\tilde{E}_{8}$ plays the role of the $2$ node in an $E_8$ residue, and $E_{8,2}$ is noncommutative.
\end{proof}

\newpage

\noindent\textbf{Claim 6:} \textit{All infinite non-affine cases are noncommutative.}

\begin{proof} The reduction arguments in this proof rely on the following fact. If Proposition~\ref{prop:test} (or Corollary~\ref{cor:comcondition}) has been used to prove noncommutativity for an $I$-parabolic Hecke algebra with Coxeter data $m_{st}$, then the $I$-parabolic Hecke algebras with Coxeter data $\overline{m}_{st}\geq m_{st}$ for all $s,t\in S$ are also noncommutative. This fact is formally stated and proved in Lemma~\ref{lem:increasebond} in the appendix (but it is clear from the way our `word arguments' work).

Suppose that $W$ is neither spherical nor affine. Let $I\subseteq S$ be spherical, and suppose that $\scH^I$ is commutative and not in the tables in the appendix. By Claim~1 we see that $|S\backslash I|=1$, and so by relabelling nodes if necessary we may assume that $I=S\backslash\{0\}$.

We will prove the following reductions based on the neighbourhood of $0$ in the Coxeter graph:

\smallskip

\noindent$\bullet$ The valency of $0$ is at most $2$, and so the diagram $I=S\backslash\{0\}$ has $1$ or $2$ connected components.

\noindent$\bullet$ If $0$ has valency $1$ then the bond number $p$ is either $3$ or $4$.

\noindent$\bullet$ If $0$ has valency $2$ then the bond numbers $p\leq q$ are $(p,q)=(3,3)$ or $(3,4)$.

\smallskip

For the first claim, suppose that $0$ has valency $4$ with bond numbers $3\leq p\leq q\leq r\leq s$. If $(p,q,r,s)=(3,3,3,3)$ then $0$ is noncommutative in a $\tilde{D}_4$ residue, and if the bond numbers are different from $(3,3,3,3)$ then we can use Lemma~\ref{lem:increasebond} to deduce noncommutativity. Thus $0$ has valency at most $3$. Suppose that $0$ has valency $3$ with bond numbers $3\leq p\leq q\leq r$. If there is at least one vertex not connected to $0$ then $0$ is noncommutative in either a $D_5$ residue, or is noncommutative by Lemma~\ref{lem:increasebond} and comparison to a ${D}_5$ residue. Thus if $0$ has valency $3$ then $S$ has exactly $4$ vertices. Suppose that there are nodes $i,j\neq 0$ which are connected. The `minimal' case is $A_1^1\times A_2^{1,2}$ (which is in the table in the appendix), and all other bond number possibilities are noncommutative by Lemma~\ref{lem:increasebond}. So suppose that $S$ has exactly $4$ nodes, and that there are no other bonds other than those which involve the $0$ node. If $(p,q,r)=(3,3,3)$ then we have a $D_4$ diagram (contradicting the assumption that $W$ is neither spherical nor affine). If $(p,q,r)=(3,3,4)$ then the $0$ node is noncommutative in $\tilde{B}_3$, and Lemma~\ref{lem:increasebond} shows that all higher bond numbers also lead to noncommutative algebras. This completes the proof of the first statement. 

To prove the second statement, if $0$ has valency $1$ with bond number at least $5$, then we can compare a suitable residue with either $\tilde{B}_{3,0}$ (if $0$ is not connected to an end vertex), or with $B_2^{1,1,1}$ (if $0$ is connected with an end vertex and there are only three vertices), or with $H_{4,4}$ (if $0$ is connected with an end vertex and there are at least four vertices) to deduce noncommutativity (applying Lemma~\ref{lem:increasebond}).

To prove the third statement, suppose that the valency of $0$ is $2$ with bond numbers $(3,n)$, $n\geq 5$, or $(m,4)$, $m\geq4$. Then we can compare an appropriate residue with $H_{3,2}$ or $\tilde{C}_{2,1}$ to deduce noncommutativity (applying Lemma~\ref{lem:increasebond}).

\medskip

The three bullet points above place severe restrictions on the Coxeter diagram $S=I\cup\{0\}$. We now eliminate each possibility using our noncommutative examples from the library in the appendix. We will give examples of the arguments used.

\smallskip

\textit{Case 1: The valency of $0$ is $1$ with $p=3$}. We consider each possible connected spherical diagram $I=S\backslash\{0\}$ and each possible way of connecting $0$ with a single bond to make~$S$. For example, suppose that $I=B_n$ with $n\geq 2$. The possible diagrams are $B_n^i$ with $i=1,\ldots,n$.
If $n=2$ then $B_2^1$ and $B_2^2$ both give $B_3$ diagrams, a contradiction, so assume that $n\geq 3$. We have $B_n^1=B_{n+1}$ and $B_n^2=\tilde{B}_{n}$ (a contradiction). Each diagram $B_n^i$ with $2<i<n<i+4$ has $0$ as a noncomutative node in a $B_{n-i+3}^3$ (and these are all in our table). If $n\geq i+4$, then we have a $B_7^3$ residue, which is noncommutative by Lemma~\ref{lem:increasebond} and comparison with $E_{8,2}$. In $B_n^n$, $n\geq 4$, the node $0$ is noncommutative in an $\tilde F_4$ residue. Thus $I=B_n$ is excluded.

\textit{Case 2: The valency of $0$ is $1$ with $p=4$}. Again we consider each diagram. For example, suppose that $I=H_3$. The diagram $H_3^{1,1}$ is in our table, and $H_3^{2,2}$ and $H_3^{3,3}$ both have $0$ as a noncommutative node in an $I_2(5)^{1,1}$ residue.

\textit{Case 3: The valency of $0$ is $2$ with $(p,q)=(3,3)$, and $I$ has one connected component}. For example suppose that $I=A_n$ with $n\geq 2$. The possibilities are $A_n^{i,j}$ with $1\leq i<j\leq n$. The case $i=1$ and $j=n$ is excluded, for it gives an $\tilde{A}_n$ diagram. By looking in a residue it suffices to show that the $0$ node is noncommutative in $A_k^{1,k-1}$ for each $k\geq 3$. The diagrams $A_3^{1,2}$ and  $A_4^{1,3}$ are in the table in the appendix. The diagram $A_5^{1,4}$ is excluded by comparing it to an $E_6$ diagram and using Lemma~\ref{lem:increasebond}. Specifically, if we decrease the bond $m_{12}=3$ in $A_5^{1,4}$ to $m_{12}=2$ then we get an $E_6$ diagram with $0$ playing the role of the (noncommutative) $3$ node. The diagram $A_6^{1,5}$ is excluded since it has $0$ as a noncommutative node in an $E_6$ residue, and for $k\geq 7$ the $A_k^{1,k-1}$ diagram is excluded since it has $0$ as the noncommutative $k-3$ node in a $D_k$ residue.

\textit{Case 4: The valency of $0$ is $2$ with $(p,q)=(3,4)$, and $I$ has one connected component}. Suppose that $0$ is connected to $i\in I$ by a single bond, and to $j\in I$ by a double bond (with $i\neq j$). The case where $i$ and $j$ are connected is excluded by Lemma~\ref{lem:increasebond} and the fact that $0$ is noncommutative in $A_2^{1,1,2}$. So suppose that $i$ and $j$ are not connected. Since $I$ is connected, $j$ is connected to some $k\in I$ with $k\neq i$. Then $0$ is noncommutative in an $F_4$ residue (incorporating $i,0,j$ and $k$) or by comparison to an $F_4$ diagram (using Lemma~\ref{lem:increasebond}).

\textit{Case 5: The valency of $0$ is $2$ with $(p,q)=(3,3)$, and $I$ has two connected components}. Let the connected components be $I_1$ and $I_2$. Suppose that $0$ is connected to $i_1\in I_1$ and $i_2\in I_2$. If there are nodes $j_1,k_1\in I_1$ connected to $i_1$ and $j_2,k_2\in I_2$ connected to $i_2$ then $0$ is a noncommutative node in a $\tilde{D}_6$ residue, or can be compared to such a vertex by Lemma~\ref{lem:increasebond}. Therefore either $i_1$ or $i_2$ is an end node.

Suppose that $i_1$ is an end node of $I_1$, and that $i_1$ is connected to $j_1\in I_1$. Assume that there exists neighbours $j_2,k_2\in I_2$ of $i_2$. Then the $0$ node is noncommutative by comparison with a $D_{7,4}$ diagram, using Lemma~\ref{lem:increasebond}. 

Suppose that there exist $j_2,k_2\in I_2$ distinct neighbours of $i_2$, and that $j_2$ has a neighbour $m_2\neq k_2$. Then the $0$ node is noncommutative by comparison with an $E_{6,3}$ diagram.

There are now 2 possibilities remaining: (i) $I_1=\{i_1\}=A_1$ and $I_2$ is a `star' with $0$ connected to the central node, or (ii) $i_1$ is an end node of $I_1$ and $i_2$ is an end node of $I_2$. (By a `star' we mean a central node with other nodes hanging off it. None of these outer nodes are connected to other outer nodes, because the diagram $I_2$ cannot have a triangle since it is spherical). Consider case (i). If any bond number of $I_2$ is $\geq 4$ then we compare with an $A_1^1\times B_3^2$ diagram. So suppose that all bonds in $I_2$ are $3$-bonds. If $I_2$ has at least four vertices then $0$ is the noncommutative in an $A_1^1\times D_4^{2}$ diagram. If $I_2$ has exactly three vertices then we have $D_5$, and if it has exactly $2$ vertices then we have $A_4$. Thus case (i) is excluded.

We are left to consider the case when $i_1$ is an end node of $I_1$ and $i_2$ is an end node of $I_2$. We consider these case by case. For example, suppose that $I_1=A_n$ and $I_2=E_m$ for $m=6,7,8$. By symmetry we can suppose that $0$ is connected to the node $1$ of $A_n$, and $0$ is not connected to the node $6$ of $E_6$. So the possibilities are $A_n^1\times E_m^k$, with $k=1,2,m$. In $A_n^1\times E_m^1$, the $0$ node is noncommutative in an $E_{8,2}$ residue. In $A_n^1\times E_m^2$, the $0$ node is noncommutative in an $E_{7,6}$ residue. Finally, in $A_n^1\times E_m^m$, $m=7,8$, the $0$ node is noncommutative in an $A_1^1\times E_m^m$ residue, which for both values of $m$ is in the table.

suppose that  $I_1=A_n$ and $I_2=E_8$. By symmetry we can suppose that $0$ is connected to the $1$ node of $A_n$, and so the possibilities are $A_n^1\times E_8^k$ with $k=1,2,8$. In $A_n^1\times E_8^1$ the $0$ node is noncommutative in an $E_8$ residue, and in $A_n^1\times E_8^2$ the $0$ node is noncommutative in an $E_7$ residue. In $A_n^1\times E_8^8$ the $0$ node is noncommutative in an $A_1^1\times E_8^8$ residue.

\textit{Case 6: The valency of $0$ is two with $(p,q)=(3,4)$, and $I$ has 2 connected components}.  Let $I_1$ and $I_2$ be the connected components. Suppose that $0$ is connected to $i\in I_1$ by a single bond, and to $j\in I_2$ by a double bond. If $|I_2|>1$ then $0$ is noncommutative in an $F_{4,2}$ residue (or can be compared to such a vertex using Lemma~\ref{lem:increasebond}). Thus $I_2=\{j\}$. If $i$ is not an end node of $I_1$ then $0$ is noncommutative in a $\tilde{B}_{4,2}$ residue (or can be compared to such a vertex). Thus $I_2=\{j\}$ and $i$ is an end node of $I_1$. So we need to consider each diagram $A_1^{1,1}\times X_n^k$ for each spherical type $X_n$ and end vertex $k$ of~$X_n$.

If $X_n$ contains a bond with bond number $\geq 4$, then $0$ is noncommutative in a $\tilde C_k$ residue (for appropriate $k$). If $X_n$ contains a vertex with degree $\geq 3$, then $0$ is
noncommutative in a $\tilde B_k$ residue (for appropriate $k$). Hence $X_n=A_n$ and we get a $B_{n+2}$ diagram.

\smallskip

Thus all infinite non-affine cases are noncommutative, and the proof of Theorem~\ref{thm:list} is complete. 
\end{proof}

Theorem~\ref{thm:commutativity} is now also clear: The `if' part is Lemma~\ref{lem:commutativity}, and the `only if' part is because, as we have seen, there are no other commutative cases other than the listed spherical cases (in which case $\pi=\mathrm{id}$) and the listed affine cases (in which case $\pi$ is opposition in the spherical residue).

\begin{appendix}

\section{Appendix}\label{app:1}

The appendix has $3$ sections. The first gives the reduction lemma that allowed us to compare diagrams to diagrams with lower bond numbers in the proof of Theorem~\ref{thm:list} (Claim 6). The second section illustrates a technique that can be used to determine if the minimal length double coset representatives of a spherical Coxeter group are involutions (this was used in Claim 2 of the proof of Theorem~\ref{thm:list}). The third section gives the tables of words that were used in the text to prove noncommutativity.

\subsection{The reduction lemma}

\begin{lemma}\label{lem:increasebond}
Let $(W,S)$ be a Coxeter system with Coxeter matrix $M=(m_{st})$. Let $I\subseteq S$. Suppose there exists $w,u,v,z\in W$ such that $w=uzv$, $u,w\in R_I$, $z\in W_I$, and $\ell(w)=\ell(u)+\ell(z)+\ell(v)$, and that there exists no $u',v',z'\in W$ with $w=v'z'u'$, $u'\leq u$, $v'\leq v$, $z'\in W_I$, and $\ell(w)=\ell(v')+\ell(z')+\ell(u')$.

Let $(\overline{W},S)$ be a Coxeter system with Coxeter matrix $\overline{M}=(\overline{m}_{st})$. Suppose that $I\subseteq S$ is spherical (for $\overline{W}$), and let $\overline{\scH}^I=\scH^I(\overline{W},S)$ be the associated $I$-parabolic Hecke algebra. If $\overline{m}_{st}\geq m_{st}$ for all $s,t\in S$ then $\overline{\scH}^I$ is noncommutative.
\end{lemma}

\begin{proof} Fix reduced expressions for $u$, $v$ and $z$. Let $\overline{u}$, $\overline{v}$ and $\overline{z}$ be the elements of $\overline{W}$ corresponding to these reduced expressions. Let $\overline{w}=\overline{u}\,\overline{z}\,\overline{v}$. Since $\overline{m}_{st}\geq m_{st}$ for all $s,t\in S$ the expressions $\overline{u},\overline{z},\overline{v}$ and $\overline{w}$ are reduced in $\overline{W}$. Furthermore, $\overline{w}$, $\overline{u}$ and $\overline{v}$ are $I$-reduced (in $\overline{W}$) and $\overline{z}\in \overline{W}_I$.

Suppose, for a contradiction, that $\overline{\scH}^I$ is commutative. By Corollary~\ref{cor:comcondition} there exist elements $\overline{u}',\overline{v}',\overline{z}'\in\overline{W}$ such that $\overline{w}=\overline{v}'\overline{z}'\overline{u}'$ with $\overline{u}'\leq u$, $\overline{v}'\leq v$, $\overline{z}'\in\overline{W}_I$, and $\ell(\overline{w})=\ell(\overline{v}')+\ell(\overline{z}')+\ell(\overline{u}')$. This gives elements $u'$, $z'$ and $v'$ in $W$ such that $w=u'z'v'$. These expressions are not necessarily reduced, and so may need to be further reduced (using the deletion condition). This leads to a contradiction. 
\end{proof}

\subsection{Involutions}

There are various ways to determine whether the minimal length double coset representatives of a spherical Coxeter group are involutions. For example \cite[Theorem~3.1]{CIK} gives a method using the representation theory of the Coxeter group. It is also possible to determine if the double coset representatives are involutions by a direct, elementary argument. Let us outline this in the most involved example~$E_{8,1}$.

\begin{prop}\label{prop:example}  Let $(W,S)$ be the Coxeter system of type $E_8$ and let $I=S\backslash\{1\}$. Each element of $R_I$ is an involution.
\end{prop}

\begin{proof} Let $\Sigma$ be the Coxeter complex of $(W,S)$ with usual $W$-distance function $\delta(u,v)=u^{-1}v$. 
Let $X$ be the set of vertices of type 1 in $\Sigma$. If $x\in X$ let $C(x)$ denote the set of all chambers of $\Sigma$ containing~$x$. For $x,y\in X$ the set $\delta(C(x),C(y))$ is a double coset $W_IzW_I$, and the $W$-distance $\delta(x,y)$ between $x$ and $y$ is defined to be the minimal length representative of this double coset. If $w\in R_I$ then (see the proof of Proposition~\ref{prop:c}) 
\begin{align}\label{eq:y}
\#\{y\in X\mid \delta(x,y)=w\}=\frac{|W_IwW_I|}{|W_I|}=\frac{|M_{I,w}||W_I|}{|W_I|}=\frac{|W_I|}{|W_{I\cap wIw^{-1}}|}.
\end{align}
 
It is known that there are exactly $10$ double cosets $W_IwW_I$ in $E_8$ (see \cite[Table~10.5]{BCN}). Let $w_0,w_1\ldots, w_{9}$ be the minimal length double coset representatives.  Fix the vertex $x_0\in X$ of type~1 contained in the chamber of $\Sigma$ corresponding to the identity element of $W$. Let $i\in\{0,1,\ldots,9\}$ be arbitrary. Put $S_i=I\cap w_iSw_i^{-1}$, and let $W_i=W_{S_i}=\langle S_i\rangle$. By (\ref{eq:y}) the number of vertices $x\in X$ with $\delta(x_0,x)=w_i$ is equal to the quotient $|W_I|/|W_i|$. The total number of vertices of type $1$ is equal to $|X|=|W|/|W_I|=2160$. Denote by $X_i$ the set of vertices $x\in X$ with $\delta(x_0,x)=w_i$. Thus $|X_0|+|X_1|+\cdots+|X_9|=|X|=2160$.

Let $w$ be the longest element in $W$. Since the opposition relation in $\Sigma$ induces the trivial permutation on $S$ (and this permutation is given by conjugation with $w$), $w$ is central in $W$. Hence if $w_i$ is an involution, then so is $ww_i$, and it interchanges $x_0$ with the unique vertex $x_i'$ opposite $x_i$, where $x_i=w_ix_0$ is the image of $x_0$ under $w_i$. Consequently if $w_i$ is an involution and if $\delta(x_0,x_i')=w_j$ then $w_j$ is also an involution. In this case we say that $w_j$ is \textit{complementary} to $w_i$. Of course it could happen that $i=j$. In this case, $x_i$ and $x_i'$ are contained in opposite chambers, and so the longest element $w$ of $W$ belongs to $w_iW_Iw_iW_I$. Since the length of the longest element in $W_I$ is 42 and since $\ell(w)=120$ this implies that $\ell(w_i)\geq 18$. 

We now apply the above to some specific values of $w_i$. We take $w_0=e$, the identity, and $w_1=s_1=1$. Thus $|X_0|=1$ and $|X_1|=|W(D_7)|/|W(A_6)|=64$, and since $\ell(w_0),\ell(w_1)<18$ we obtain complementary involutions $w_9$ and $w_8$, respectively, with $|X_9|=1$ and $|X_8|=64$. Now put $w_2=13425431$ (which is obtained by considering the residue of a vertex of type 6). The element $w_2$ maps the generators $(3,4,2,5,7,8)$ to $(3,4,5,2,7,8)$, and so one calculates that $|X_2|=|W(D_7)|/|W(D_4\times A_2)|=280$. Since $\ell(w_2)=8<18$ we have a complementary involution $w_7\neq w_2$ with $|X_7|=280$. So far we have accounted for $2(1+64+280)=690$ of the total $2160$ type~$1$ vertices.

In the residue of an element of type $8$ (which is a Coxeter system of type $E_7$) we find the involutive minimal length double coset representative $w_3=13425463576452431$, which maps the generators $(2,4,5,6,7)$ to $(7,6,5,4,2)$. Consequently $|X_3|=|W(D_7)|/|W(A_5)|=448$. As $\ell(w_3)=17<18$ we have another involution $w_6$ with $|X_6|=448$, accounting for $690+2\times 448=1586$ of the $2160$ vertices. Finally we can consider, in each of the 14 residues of type $E_7$ through $x_0$, the element of type 1 opposite $x_0$. This gives rise to another involutive double coset representative $w_4$, with $|X_4|=|W(D_7)|/|W(D_6)|=14$. This one must be self-complementary, as otherwise the unique missing class $X_5$ would also contain $14$ elements and the total number of vertices does not add up to $2160$. Indeed we calculate that $|X_5|=560$. Hence $w_5$ is also self-complementary. But what is more important, it must also be an involution as otherwise $w_5^{-1}$ is a different minimal double coset representative, contradicting the fact that we only have 10 of these. Hence all minimal coset representatives are involutions.
\end{proof}

\subsection{Tables of words to prove noncommutativity}

\textit{Conventions:} We use standard Bourbaki labelling for the spherical and affine types \cite[Plates I--IX]{bourbaki}. The cases $H_3$ and $H_4$ are not given an explicit labelling in Bourbaki. We adopt the labelling of $H_3$ with $m_{12}=3$ and $m_{23}=5$, and of $H_4$ with $m_{12}=m_{23}=3$ and $m_{34}=5$. 

Each word is of the form $w=uw_Is_i$, where $I=S\backslash\{i\}$. We also list the index $k$ used in the argument of Proposition~\ref{prop:test}. The cases where a slight modification of Proposition~\ref{prop:test} is required are labelled by~$(*)$. The precise details for these cases are given in Claim~4 of Section~\ref{section:proof}.

$$
\begin{array}{|c||c|c|c|}
\hline
\textrm{Spherical cases}& u & w_I & k \\
\hline\hline
 D_{n,i} ,  \frac{n}{2}<i<n-1 &\textrm{see below} & \textrm{see below} &n\\
 \hline
 E_{6,5} & 542345 & 1634 & 6 \\
\hline
 E_{7,6} & 65423456 & 17345 & 7 \\
\hline
 E_{8,7} & 7654234567 & 183456 & 8 \\
\hline
 E_{8,2} & 245678345672 & 456345134 & 1 \\
\hline
 F_{4,2} & 232 & 431 & 1 \\
\hline
 H_{3,2} & 232 & 31 & 1 \\
\hline
 H_{4,2} & 23432 & 431 & 1 \\
\hline
 H_{4,4} & 434323434 & 123 & 1(*) \\
\hline
\end{array}
  $$

\noindent The $D_{n,i}$ word (with $n/2<i<n-1$) is 
\begin{align*}
u&=[i(i-1)\cdots(2i-n+1)][(i+1)i\cdots(2i-n+2)]\cdots[(n-1)(n-2)\cdots i]\\
w_I&=\begin{cases}
[n(n-2)(n-3)\cdots(i+1)][12\cdots(i-1)]&\textrm{if $i/2<i<n-2$}\\
n12\cdots (n-3)&\textrm{if $i=n-2$}.
\end{cases}
\end{align*}

 $$
\begin{array}{|c||c|c|c|}
\hline
\textrm{Affine cases}& u & w_I & k \\
\hline\hline
\tilde{B}_{n,i},  1<i<n-1 & i\cdots 320123\cdots i & (i+1)\cdots n(n-1)\cdots(i+1) & i+1 \\
\hline
\tilde{B}_{n,n} & [n\cdots1][n\cdots 2]\cdots [n(n-1)][n] & 023\cdots (n-2)(n-1) &0\\
\hline
\tilde{C}_{n,i} ,  1\leq i< n & i\cdots3210123\cdots i & (i+1)\cdots n(n-1)\cdots (i+1) & i+1 \\
\hline
\tilde{D}_{n,i} ,  1<i<n-1 & i\cdots 320123\cdots i & (i+1)\cdots n(n-2)\cdots(i+1) & i+1 \\
\hline
\tilde{E}_{7,2} & 245341031245342 & 65764534 & 7 \\
\hline
\tilde{E}_{8,1} & 134562453413245676805432456781 &345672456345243  &0(*)\\
\hline
 \tilde{E}_{8,8} & 876542345678 & 1034567 & 0 \\
\hline
 \tilde{F}_{4,1} & 12321 & 4320 & 0 \\
\hline
 \tilde{F}_{4,4} & 43231234 & 3231230123 & 0(*) \\
\hline
  \tilde{G}_{2,1} & 212 & 01 & 0 \\
\hline
  \tilde{G}_{2,2} & 12121 & 02 & 0 \\
\hline
\end{array}
  $$

$$
\begin{array}{|c||c|c|c|}
\hline
\textrm{Infinite non-}& u & w_I & k \\
\textrm{affine cases} & & &\\
\hline\hline
A_2^{1,1,2} &010 &21 &2\\
\hline
A_3^{1,2} &0210 & 2312 & 3 \\
\hline
A_4^{1,3} &032430   &123 &1 \\
\hline
B_2^{1,2} &01210 &212 &2(*) \\
\hline
B_2^{1,1,1}&010 &121 &2\\
\hline
B_3^{1,3} &03230 &12321  &1 \\
\hline
B_3^{3,3} &0323032303230 &1323  &1 \\
\hline
B_4^3 & 03430&234123 & 1(*)\\
\hline
B_5^3 &032430  &1234543 &5\\
\hline
B_6^3 &032430 &123456543 &5\\
\hline
D_4^{1,4} &01240 &123421 &3 \\
\hline
D_5^3 & 03243120 & 3543 &5\\ 
\hline
D_5^{1,5} & 05342350 &12345321 &1 \\
\hline
D_6^3 &03243120 & 346543 & 6 \\
\hline
D_6^{1,6} &06453460 &1234564321 &1 \\
\hline
D_7^3 &03243120 & 34576543 & 7\\
\hline
E_6^3 &0345243013452430 &61345243 & 6\\
\hline
E_7^2 &02435420 &65431243524654376542 &7 \\
\hline
E_7^6 &06543245607654324560 & 1765432456 &1 \\
\hline
E_8^1 & 0134254310 & 654234567813425436542765431 &2(*) \\
\hline
E_8^7 & 076543245670876543245670  & 187654324567 &1  \\
\hline
F_4^2 &02320& 1234232 &4 \\
\hline
F_4^{1,1} &0123210123210 &412321  &4 \\
\hline
F_4^{1,4} &0432340   & 12321&1 \\
\hline
H_3^2 &02320 &32132 &1 \\
\hline
H_3^3 &03230 &2321323 &1 \\
\hline
H_3^{1,1} & 010&232132321 &2(*) \\
\hline
H_4^1 & 012343210 & 43423412324341234321 & 4(*)\\
\hline
H_4^2 & 0234320& 4342312& 1\\
\hline
I_2(5)^{1,1} &010 &2121 &2(*)\\
\hline
I_2(7)^{1} &012120 &12121 &1(*)\\
\hline
A_1^1\times A_2^{1,2} &0120 &11' &1' \\
\hline
A_1^1\times F_4^1 & 0123210 & 1'4321 &1' \\
\hline
A_1^1\times H_3^1 & 0123210 & 1'321 &1' \\
\hline
A_1^1 \times B_3^2 & 02320 & 121' &1'\\
\hline
A_1^1\times D_4^2 & 023420 &1'12 &1' \\
\hline
A_1^1\times E_7^7 & 076542345670 & 11'34567& 1'\\
\hline
A_1^1\times E_8^8 & 08765423456780 & 11'345678 & 1' \\
\hline
\end{array}
$$
(The node of the $A_1$ component in the final $7$ composite cases is labelled by~$1'$).

\end{appendix}

\end{document}